\documentclass[11pt,reqno]{amsart}

\usepackage[utf8]{inputenc}
\usepackage{setspace}
\usepackage{geometry}
\usepackage{enumerate}
\usepackage{enumitem, xcolor, amssymb,latexsym,amsmath,bbm}
\usepackage{mathtools}
\usepackage[mathscr]{euscript}
\usepackage{amsmath}
\usepackage{amssymb}
\usepackage{enumitem}

\usepackage{tikz}
\usepackage{wrapfig}
\usepackage{enumerate}
\usepackage{graphicx}
\usepackage{subfigure}
\usepackage{esint}
\usetikzlibrary{arrows.meta}
\usepackage{mathrsfs}
\usetikzlibrary{decorations.markings}
\usepackage[colorlinks=true,citecolor=blue,
linkcolor=blue,urlcolor=blue]{hyperref}

\newtheorem{theorem}{Theorem}
\newtheorem{theorem*}{Theorem}
\newtheorem{lemma}{Lemma}

\newtheorem{property}{Property}

\newcommand{\RR}{\mathbb{R}}

\newcommand{\ZZ}{\mathbb{Z}}

\newcommand{\HH}{\mathbb{H}}

\numberwithin{equation}{section}
\numberwithin{theorem}{section}
\numberwithin{lemma}{section}
\numberwithin{example}{section}
\allowdisplaybreaks

\makeatletter
\@namedef{subjclassname@2020}{\textup{2020} Mathematics Subject Classification}
\makeatother

\begin{document}
    
\title[Fractional Hardy's inequality for half spaces in $\HH^n$]{Fractional Hardy's inequality for half spaces in the Heisenberg group}

    \author{Rama Rawat}
    \email{rrawat@iitk.ac.in}

    \author{Haripada Roy}
    \email{haripada@iitk.ac.in}

    \address{Indian Institute of Technology Kanpur, India}

    \date{}

    \smallskip
\begin{abstract}
We establish the following fractional Hardy's inequality
$$\int_{\HH^n_+}\frac{|f(\xi)|^p}{x_1^{sp}|z|^\alpha}d\xi\leq C\int_{\HH^n_+}\int_{\HH^n_+}\frac{|f(\xi)-f(\xi')|^p}{d({\xi}^{-1}\circ \xi')^{Q+sp}|z'-z|^\alpha}d\xi'd\xi,\ \ \forall\,f\in C_c(\HH^n_+)$$
for the half space $\HH^n_+:=\{\xi=(z,t)=(x_1,x_2,\ldots, x_n, y_1,y_2,\ldots,y_n,t)\in\HH^n:x_1>0\}$ in the Heisenberg group $\HH^n$ under the conditions $sp>1$ and $\alpha\geq (2n+sp)/2$. We also provide an alternate proof of a fractional Hardy's inequality in $\HH^n$ established in an earlier work.
\end{abstract}
    
    %\pagenumbering{roman}
    \maketitle

        %\dedicatory{}
        %\tableofcontents
        %\pagenumbering{arabic}
        %\setcounter{page}{1}
        %\pagestyle{myheadings}

\section{Introduction}
Hardy's inequality on the Heisenberg group was first established by Garofalo and Lanconelli in \cite{Garofalo}. They established the following result: Let $u\in C_c^{\infty}(\mathbb{H}^n\setminus \{0\})$. Then
\begin{equation}\label{eq garo}
\int_{\mathbb{H}^n}\left(\frac{|z|}{d(\xi)}\right)^2\frac{|u|^2}{d(\xi)^2}\leq \Big(\frac{2}{Q-2}\Big)^2\int_{\mathbb{H}^n}|\nabla_{\mathbb{H}^n}u|^2d\xi.
\end{equation}
Many generalizations of this inequality were obtained by various authors. For these, we refer to the works in \cite{Adimurthi_Sekar_2006,D'Ambrosio,Pengcheng3,Xiao}. However, our literature search for fractional Hardy type inequalities for subdomains, or specifically for the half spaces of $\HH^n$, came to nothing. This work is an attempt in this direction. Unlike the Euclidean case, the non-isotropic guage $d(\xi)$ in $\HH^n$ forces that the distance function behaves differently for the hyperplanes 
$\HH^n_{+}$ and $\HH^n_{+,t}$ which are defined later.
Following closely the technique in \cite{Dyda}, we have been able to establish a fractional Hardy type inequality for the half space $\HH^n_+$, although this approach does not allow us to compute the corresponding sharp constant, in this case. However, we plan to address this in our next work. The corresponding questions for the hyperplane $\HH^n_{+,t}$ remain open due to the lack of a closed form of the corresponding distance function for this hyperplane. We come back to this point in a while, and also in the Appendix in Section 4.\smallskip

To recall , in 1925, G. H. Hardy established a fundamental integral inequality for non-negative measurable functions on $\RR_+$, known as the ``Hardy's inequality". The generalization of Hardy's inequality for the Euclidean space $\RR^n$ ($n\geq2$) is of the form 
\begin{equation}
\int_{\RR^n}\frac{|u(x)|^p}{|x|^p}dx\leq \left|\frac{p}{n-p}\right|^p\int_{\RR^n}|\nabla u(x)|^pdx,
\end{equation}
which holds for all $u\in C_c^\infty(\RR^n)$ if $1<p<n$ and for all $u\in C_c^\infty(\RR^n\setminus\{0\})$ if $p>n$. The constant $|\frac{p}{n-p}|^p$ was shown to be sharp in \cite{Azorero} for the case $1<p<n$ and in \cite[pp. 9]{Alexander} for the case $p>n$. Let $\Omega$ be a domain in $\RR^n$, $n\geq1$, with non-empty boundary and let us define the the distance function from the boundary of $\Omega$ by $\delta_\Omega(x)=\mathrm{dist}(x,\partial\Omega)=\inf\{|x-y|:y\in\partial\Omega\}$. The following variant of Hardy's inequality was established by Alano Ancona in \cite{Ancona} for a sufficiently regular domain $\Omega$:
\begin{equation}\label{bdry har}
\int_\Omega\frac{|u(x)|^2}{\delta_\Omega(x)^2}dx\leq C_\Omega\int_\Omega|\nabla u(x)|^2dx\ \ \forall\,u\in C_c^\infty(\Omega).
\end{equation}
Later on, this inequality was generalized by several authors for various domains, such as bounded convex domains, bounded domains with smooth or Lipschitz boundaries, the upper half-spaces etc. and the corresponding best constants were also discussed in e.g., \cite{Hoffmann,Tanya,Tidblom,Tidblom2}.\smallskip

Analogous versions of these inequalities on fractional Sobolev spaces have also been studied for several domains. We refer \cite{Bourgain,Hitchhiker,Frank,Mazya} for introduction of the fractional Sobolev spaces, its connection with the classical spaces and for the Hardy's inequality for fractional Sobolev spaces. Let $ p>0,$  $0<s<1$ and let $\Omega$ be as before a domain in $\RR^n$, $n\geq1$. In \cite{Dyda}, B. Dyda established the fractional Hardy's inequality
\begin{equation}\label{eq Dyda}
\int_\Omega\frac{|u(x)|^p}{\delta_\Omega(x)^{sp}}dx\leq C\int_\Omega\int_\Omega\frac{|u(x)-u(y)|^p}{|x-y|^{n+sp}}dydx\ \ \forall\,u\in C_c^\infty(\Omega),
\end{equation}
for the following cases:\\
(1) $\Omega$ is a bounded Lipschitz domain and $sp>1$;\\
(2) $\Omega$ is the complement of a bounded Lipschitz domain, $sp\neq1$ and $sp\neq n$;\\
(3) $\Omega$ is a domain above the graph of a Lipschitz function $\RR^{n-1}\rightarrow\RR$ and $sp\neq1$;\\
(4) $\Omega$ is the complement of a point, and $sp\neq n$.\smallskip

Inequality \eqref{eq Dyda} for the critical case $sp=1$ is addressed in \cite{Purbita_Adi_Pro} for dimension one and in \cite{Adi_Vivek_Pro} for higher dimensions. L. Brasco and E. Cinti established another variant of fractional Hardy's inequality for convex domain $\Omega\neq\RR^n$ in \cite{BraCin} without imposing any restriction on the quantity $sp$. Inequality \eqref{eq Dyda} for the upper half space $\RR^n_+=\{x\in\RR^n:x_n>0\}$ and the corresponding sharp constant
\begin{equation}\label{best constant}
C_{n,p,s}:=2\pi^{\frac{n-1}{2}}\frac{\Gamma((1+sp)/2)}{\Gamma((n+sp)/2)}\int_0^1\left|1-r^{(sp-1)/p}\right|^p\frac{dr}{(1-r)^{1+sp}}
\end{equation}
was first established by Bogdan and Dyda in \cite{Dyda2} for the case $p=2$ under the condition $sp\neq1$. The results for general $p$ were later extended by Frank and Seiringer in \cite{Frank2}.\smallskip

In this paper our focus will be on inequality like \eqref{eq Dyda} for the half spaces in the Heisenberg group. Before going to the main results, let us recall needed background on the Heisenberg group. Let $n\geq1$. The Heisenberg group $\HH^n$ is defined by
$$\HH^n:=\{\xi=(z,t)=(x,y,t):z=(x,y)\in \RR^n\times \RR^n\ \mathrm{and}\ t\in \RR\},$$
with the following group law:
$$\xi\circ\xi'=(x+x',y+y',t+t'+2\langle y,x'\rangle-2\langle x,y'\rangle),$$
where $\xi=(x,y,t)$, $\xi'=(x',y',t')\in \HH^n$ and $\langle\cdot,\cdot\rangle$ denotes the usual Euclidean inner product in $\RR^n$. One can easily check that $0\in\HH^n$ is the identity element and $-\xi=(-z,-t)$ is the inverse of $\xi\in\HH^n$. A basis for the left invariant vector fields is given by
$$X_i=\frac{\partial}{\partial x_i}+2y_i\frac{\partial}{\partial t},\ \ 1\leq i\leq n,$$
$$Y_i=\frac{\partial}{\partial y_i}-2x_i\frac{\partial}{\partial t},\ \ 1\leq i\leq n,$$
$$T=\frac{\partial}{\partial t}.$$
For any function $u$ of $C^1$ class, $\nabla_{\HH^n}(u):=\left(X_1(u),\ldots,X_n(u),Y_1(u),\ldots,Y_n(u)\right)$ is the associated Heisenberg sub-gradient of $u$ and we define
$$|\nabla_{\HH^n}(u)|^2=\sum_{i=1}^{n}\left(|X_i(u)|^2+|Y_i(u)|^2\right).$$
For $\xi=(x,y,t)=(z,t)\in\HH^n$, $d(\xi)=\left((|x|^2+|y|^2)^2+t^2\right)^{\frac{1}{4}}=\left(|z|^4+t^2\right)^{\frac{1}{4}}$ is called the Koranyi-Folland non-isotropic gauge. The left-invariant Haar measure on $\HH^n$ is the Lebesgue measure on $\RR^{2n+1}$. For any measurable set $A\subset\HH^n$, we define by $|A|$, the Lebesgue measure of $A$. The dilations on $\HH^n$ are defined by
$$D_r\xi:=(rx,ry,r^2t),\ \ \mathrm{for}\ \xi=(x,y,t)\in \mathbb{H}^n\ \mathrm{and}\ r>0.$$
The determinant of the Jacobian matrix of $D_r$
is $r^Q$, where $Q=2n+2,$ is called the homogeneous dimension of $\mathbb{H}^n$. We define the ball centered at $\xi$ with radius $r$ in $\HH^n$ as
\begin{equation}\label{set B_r}
B(\xi,r)=\{\xi'=(z',t'):d({\xi}^{-1}\circ \xi')<r\}.
\end{equation}
and the set
\begin{equation}\label{set B_z_r}
B_z(\xi,r)=\{\xi'=(z',t'):|z'-z|<r\},
\end{equation}
is an infinite cylinder of radius $r$, parallel to the $t$ direction, whose axis passes through $\xi$.\smallskip

The distance function for a domain $\Omega\subset\HH^n$ with non-empty boundary is defined by
$$\delta_\Omega(\xi)=\begin{cases}
\underset{\xi'\in\partial\Omega}{\mathrm{inf}}d(\xi^{-1}\circ\xi') & \text{if $\xi\in\Omega$}\\
0 & \text{if $\xi\in\HH^n\setminus\Omega$}.
\end{cases}$$
Let us define the following half spaces in $\HH^n$:
$$\HH^n_{+}:=\{\xi=(z,t)=(x,y,t)=(x_1,x_2,\ldots,x_n, y_1,y_2,\ldots,y_n,t)\in\HH^n:x_1>0\}$$
and
$$\HH^n_{+,t}:=\{\xi=(z,t)=(x,y,t)\in\HH^n:t>0\}.$$
It is easy to see that $\delta_{\HH^n_{+}}(\xi)=x_1$ by choosing $\xi'=(0,x_2,\ldots,x_n,y,t-2x_1y_1)\in\partial\HH^n_+$, whereas the distance function $\delta_{\HH^n_{+,t}} (\xi) \leq \sqrt{t}$ for the half space $\HH^n_{+,t}$. Indeed, the distance from the boundary of $\HH^n_{+,t}$ is $\sqrt{t}$ for $\xi=(0,0,t)\in \HH^n_{+,t} $ but the value could be strictly less than $\sqrt{t}$ for $\xi=(z,t)$ in some subsets of $\HH^n_{+,t}$, see the Appendix. Moreover, we do not know a closed form for the distance function in this case.\smallskip

The inequality \eqref{eq garo} of Garofalo and Lanconelli in \cite{Garofalo} has been followed by establishing versions of Hardy's inequalities for subdomains in the Heisenberg group, much like inequalities of the form \eqref{bdry har} for the Euclidean setting. This includes the work in \cite{Pengcheng} where the authors prove a version of such inequalities for half spaces $\HH^n_{+},$ with a remainder term. Likewise in \cite{Luan}, the authors prove a version for half space $\HH^n_{+,t},$ with the difference that the distance function is being replaced by the function $\sqrt{t}.$ The work in \cite{Larson}, contains a version of these Hardy inequalities for convex domains in $\HH^n,$
convexity being defined in the Euclidean sense and the distance function is being replaced by a weighted Euclidean distance.

Let $1\leq p<\infty$, $s\in(0,1)$ and $\alpha\geq0$. For any open set $\Omega\subset\HH^n$, we consider the following fractional Sobolev space
$$W^{s,p,\alpha}(\Omega):=\left\{f\in L^p(\Omega):[f]_{s,p,\alpha,\Omega}:=\left(\int_{\Omega}\int_{\Omega}\frac{|f(\xi)-f(\xi')|^p}{d({\xi}^{-1}\circ \xi')^{Q+sp}|z-z'|^{\alpha}}d\xi' d\xi\right)^{\frac{1}{p}}<\infty\right\},$$
the norm is defined by $\lVert f\rVert_{s,p,\alpha,\Omega}^p=\lVert f\rVert_{L^p(\Omega)}^p+[f]_{s,p,\alpha,\Omega}^p$. The space $W_0^{s,p,\alpha}(\Omega)$ is the closure of $C_c^\infty(\Omega)$ under the norm $\lVert f\rVert_{s,p,\alpha,\Omega}$. The space $W_0^{s,p,\alpha}(\HH^n)$ was first introduced in \cite{Adi-Mallick}. The following fractional Hardy type inequality for the function $f\in W_0^{s,p,\alpha}(\HH^n)$ was established in \cite{Adi-Mallick} under the condition $sp+\alpha<Q$ together with $sp>2$:
\begin{equation}\label{eq adi}
\int_{\HH^n}\frac{|f(\xi)|^p}{d(\xi)^{sp}|z|^{\alpha}}d\xi\leq C_{n,p,s,\alpha}\int_{\HH^n}\int_{\HH^n}\frac{|f(\xi)-f(\xi')|^p}{d({\xi}^{-1}\circ \xi')^{Q+sp}|z'-z|^{\alpha}}d\xi'd\xi.
\end{equation}
The case $\alpha=0$ was established in \cite{Suragan}. Recently, the range of the indices has been improved as a particular case of fractional Caffarelli-Kohn-Nirenberg type inequalities in \cite{Hari}, which is as follows:\\
(i) The inequality \eqref{eq adi} holds for the functions of $W_0^{s,p,\alpha}(\HH^n)$ if, either $sp+\alpha<Q-2$, or $sp+\alpha<Q$ with $sp\geq2$.\\
(ii) The inequality \eqref{eq adi} holds for the functions of $W_0^{s,p,\alpha}(\HH^n\setminus\{\xi\in\HH^n:z=0\})$ if $sp+\alpha>Q-2$.\smallskip

Our main result is the following inequality.

\begin{theorem}\label{thm 1.1}
Let $p\geq1$ and $s\in(0,1)$ be such that $sp>1$ and $\alpha\geq (Q-2+sp)/2$. Then there exists a constant $C>0$ depending only on $n$, $p$, $s$ and $\alpha$, such that
\begin{equation}\label{eq main 1}
\int_{\HH^n_+}\frac{|f(\xi)|^p}{x_1^{sp}|z|^\alpha}d\xi\leq C\int_{\HH^n_+}\int_{\HH^n_+}\frac{|f(\xi)-f(\xi')|^p}{d({\xi}^{-1}\circ \xi')^{Q+sp}|z'-z|^\alpha}d\xi'd\xi\ \ \forall\,f\in C_c(\HH^n_+).
\end{equation}
\end{theorem}\smallskip

As in other Hardy type inequalities on $\HH^n$, here the term $|z|$ on the left hand side occurs naturally from the structure of the Heisenberg group $\HH^n$.\smallskip

Again using Dyda's technique, we provide an alternate proof of the improved version of the fractional Hardy-type inequality \eqref{eq adi}, which was presented as a particular case of Theorem 1.1 in \cite{Hari}. More precisely, we prove the following.

\begin{theorem}\label{thm 1.2}
Let $p\geq1$, $s\in(0,1)$ and $\alpha\geq0$.\\
\emph{(i)} If $sp+\alpha<Q-2$, then \eqref{eq adi} holds for all $f\in C_c(\HH^n)$.\\
\emph{(ii)} If $sp+\alpha>Q-2$, then \eqref{eq adi} holds for all $f\in C_c(\HH^n\setminus\{\xi\in\HH^n:z=0\})$.
\end{theorem}\smallskip

In the proof of Euclidean results, Dyda used  crucially the decomposition shown in Figure 1 for the set $\{x=(x_1,x_2,\ldots,x_n)\in\RR^n:0<x_1<\rho\ \mathrm{and}\ |x_i|<\rho/2,\ i=2,\ldots,n\},$ where $\rho>0$ is arbitrary. In this figure, the base in $\RR^{n-1}$ (i.e., the projection onto $x_2,x_3, \ldots, x_n$ coordinates) is divided into dyadic cubes and the height in $x_1$-direction is scaled appropriately to allow for shifting of integrals between the cubes which have the same base in $x_2,x_3, \ldots, x_n$ coordinates, but are at different $x_1$-heights. 
We will also be using a similar decomposition adapted to the Heisenberg group case; see Section 2.

\begin{center}
\begin{tikzpicture}
\draw[thick] (-5,0) -- (5,0) node[anchor=north west] {$\RR^{n-1}$};
\draw[thick,->] (0,0) -- (0,4.5) node[anchor=south east] {$x_1$};
\draw (-4,0) rectangle (4,4);
\draw (-4,2) -- (4,2);
\draw (-4,1) -- (4,1);
\draw (-4,0.5) -- (4,0.5);
\draw (-4,0.25) -- (4,0.25);
\draw[dashed] (-2,2) -- (-2,0.25);
\draw[dashed] (2,2) -- (2,0.25);
\draw[dashed] (-3,1) -- (-3,0.25);
\draw[dashed] (-1,1) -- (-1,0.25);
\draw[dashed] (1,1) -- (1,0.25);
\draw[dashed] (3,1) -- (3,0.25);
\draw[dashed] (-3.5,0.5) -- (-3.5,0.25);
\draw[dashed] (-2.5,0.5) -- (-2.5,0.25);
\draw[dashed] (-1.5,0.5) -- (-1.5,0.25);
\draw[dashed] (-0.5,0.5) -- (-0.5,0.25);
\draw[dashed] (0.5,0.5) -- (0.5,0.25);
\draw[dashed] (1.5,0.5) -- (1.5,0.25);
\draw[dashed] (2.5,0.5) -- (2.5,0.25);
\draw[dashed] (3.5,0.5) -- (3.5,0.25);

\draw (-4.25,0.25) -- (-4.25,4);
\draw (-4.3,4) -- (-4.2,4);
\draw (-4.3,2) -- (-4.2,2);
\draw (-4.3,1) -- (-4.2,1);
\draw (-4.3,0.5) -- (-4.2,0.5);
\draw (-4.3,0.25) -- (-4.2,0.25);

\draw (-4.5,3) node{$\frac{\rho}{2}$};
\draw (-4.5,1.5) node{$\frac{\rho}{4}$};
\draw (-4.5,0.75) node{$\frac{\rho}{8}$};

\draw (-4,-0.25) -- (4,-0.25);
\draw (-4,-0.2) -- (-4,-0.3);
\draw (0,-0.2) -- (0,-0.3);
\draw (4,-0.2) -- (4,-0.3);

\draw (-2,-0.5) node{$\rho/2$};
\draw (2,-0.5) node{$\rho/2$};
\end{tikzpicture}
\end{center}
$$\mathrm{Figure}\ 1.\ n=2$$

\smallskip

The paper is organized as follows. In Section 2, we introduce a special type of decomposition of $\HH^n_+$ and use it to prove Lemma \ref{lemma 2.2}, from which Theorem \ref{thm 1.1} follows. Section 3 contains the proof of Theorem \ref{thm 1.2}. An estimate of the distance function for the half space  $\HH^n_{+,t}$ is discussed in the Appendix.

\section{Proof of the fractional Hardy's inequality in half space}
For any fixed $\xi\in\HH^n_+$ and $R_1,R_2>0$, let us define the set
$$\Sigma_{R_1,R_2}(\xi)=\{\xi'\in\HH^n:d({\xi}^{-1}\circ \xi')<R_1x_1^{1/2}|z|^{1/2}\ \ \mathrm{and}\ \ |z'-z|<R_2x_1\}.$$
Let $\mathscr{D}\subset \HH^n_+$ be any open set and $f\in C_c(\HH^n_+)$. For $R_1,R_2$ and $S>0$ we define the following set:
\begin{equation}\label{set F1}
F=F(f,\mathscr{D},R_1,R_2,S)=\left\{\xi\in\mathscr{D}:|f(\xi)|^p>\frac{2^{p+1}}{Sx_1^{Q-1}|z|}\int_{\Sigma_{R_1,R_2}(\xi)\cap\mathscr{D}}|f(\xi)-f(\xi')|^pd\xi'\right\}.
\end{equation}
Note that $F\subset\mathrm{supp}\,f$, or in particular, $\mathrm{dist}(F,\partial\HH^n_+)>0$. The role of the set $\mathscr{D}$ is to localise the setting. The following properties of the set $F$ are useful to prove our first result.
\begin{property}\label{property 1}
Let $\alpha\geq\frac{Q-2+sp}{2}$. Then we have
$$\int_{\mathscr{D}\setminus F}\frac{|f(\xi)|^p}{x_1^{sp}|z|^\alpha}d\xi\leq \frac{2^{p+1}R_1^{Q+sp}R_2^{\alpha}}{S}\int_{\mathscr{D}\setminus F}\int_{\mathscr{D}}\frac{|f(\xi)-f(\xi')|^p}{d({\xi}^{-1}\circ \xi')^{Q+sp}|z'-z|^\alpha}d\xi' d\xi.$$
\end{property}
\begin{proof}
Let $\xi\in\mathscr{D}\setminus F$. We have
\begin{multline*}
\int_{\mathscr{D}}\frac{|f(\xi)-f(\xi')|^p}{d({\xi}^{-1}\circ \xi')^{Q+sp}|z'-z|^\alpha}d\xi'\geq \int_{\Sigma_{R_1,R_2}(\xi)\cap\mathscr{D}}\frac{|f(\xi)-f(\xi')|^p}{(R_1x_1^{1/2}|z|^{1/2})^{Q+sp}(R_2x_1)^\alpha}d\xi'\\
\geq\frac{S}{2^{p+1}R_1^{Q+sp}R_2^{\alpha}}\left(\frac{|z|}{x_1}\right)^{\alpha-\frac{Q-2+sp}{2}}\frac{|f(\xi)|^p}{x_1^{sp}|z|^\alpha}\geq\frac{S}{2^{p+1}R_1^{Q+sp}R_2^{\alpha}}\frac{|f(\xi)|^p}{x_1^{sp}|z|^\alpha}.
\end{multline*}
Integration over $\mathscr{D}\setminus F$ completes the proof.
\end{proof}
\begin{property}\label{property 2}
Let $\xi\in F$. For $E\subset \Sigma_{R_1,R_2}(\xi)\cap\mathscr{D}$ we define the set
$$E^*=E^*(\xi):=\left\{\xi'\in E:\frac{1}{2}|f(\xi)|\leq |f(\xi')|\leq \frac{3}{2}|f(\xi)|\right\}.$$
If $|E|\geq Sx_1^{Q-1}|z|$, then $|E^*|\geq |E|-\frac{1}{2}Sx_1^{Q-1}|z|\geq \frac{1}{2}|E|$.
\end{property}
\begin{proof}
It is easy to observe that
$$E^*(\xi)\supseteq \left\{\xi'\in E:\frac{||f(\xi)|-|f(\xi')||}{|f(\xi)|}\leq\frac{1}{2}\right\}.$$
Therefore
\begin{multline*}
|E^*|\geq |E|-\left|\left\{\xi'\in E:\frac{||f(\xi)|-|f(\xi')||}{|f(\xi)|}>\frac{1}{2}\right\}\right|\\
\geq |E|-2^p\int_{\left\{\xi'\in E:\frac{||f(\xi)|-|f(\xi')||}{|f(\xi)|}>\frac{1}{2}\right\}}\frac{||f(\xi)|-|f(\xi')||^p}{|f(\xi)|^p}d\xi'\\
\geq |E|-\frac{2^p}{|f(\xi)|^p}\int_E|f(\xi)-f(\xi')|^pd\xi'.
\end{multline*}
The rest of the proof follows from the assumptions on $\xi$ and $E.$
\end{proof}
\begin{property}\label{property 3}
Let $E_1\subset\mathscr{D}$. If $E_2$ is any set satisfying the following: $E_2\subset \Sigma_{R_1,R_2}(\xi)\cap\mathscr{D}$ and $|E_2|\geq Sx_1^{Q-1}|z|$ for all $\xi\in E_1$, then
$$\int_{E_1\cap F}\frac{|f(\xi)|^p}{x_1^{sp}|z|^\alpha}d\xi\leq 2^{p+1}\frac{|E_1|}{|E_2|}\frac{\sup\{{x_1'}^{sp}|z'|^\alpha:\xi'\in E_2\}}{\inf\{x_1^{sp}|z|^\alpha:\xi\in E_1\}}\int_{E_2}\frac{|f(\xi')|^p}{{x_1'}^{sp}|z'|^\alpha}d\xi'.$$
\end{property}
\begin{proof}
Assume $E_1\cap F\neq\phi$. For any fixed $\eta>1$ we pick $\xi_\eta\in E_1\cap F$ such that $\underset{E_1\cap F}{\sup}|f|\leq \eta|f(\xi_\eta)|$. It follows from the definition that $\frac{1}{2}|f(\xi_\eta)|\leq |f(\xi')|$ for all $\xi'\in E_2^*(\xi_\eta)$. Therefore
\begin{multline*}
\int_{E_1\cap F}\frac{|f(\xi)|^p}{x_1^{sp}|z|^\alpha}d\xi\leq |E_1\cap F|\cdot\frac{\underset{E_1\cap F}{\sup}|f|^{p}}{\inf\{x_1^{sp}|z|^\alpha:\xi\in E_1\cap F\}}\\
\leq \frac{|E_1\cap F|}{|E_2^*(\xi_\eta)|}\frac{\sup\{{x_1'}^{sp}|z'|^\alpha:\xi'\in E_2^*(\xi_\eta)\}}{\inf\{x_1^{sp}|z|^\alpha:\xi\in E_1\cap F\}}\cdot\eta^p\int_{E_2^*(\xi_\eta)}\frac{|f(\xi_\eta)|^p}{{x'_1}^{sp}|z'|^\alpha}d\xi'\\
\leq 2^{p+1}\eta^p \frac{|E_1|}{|E_2|}\frac{\sup\{{x_1'}^{sp}|z'|^\alpha:\xi'\in E_2\}}{\inf\{x_1^{sp}|z|^\alpha:\xi\in E_1\}}\int_{E_2}\frac{|f(\xi')|^p}{{x_1'}^{sp}|z'|^\alpha}d\xi'.
\end{multline*}
Here we use the fact that $|E_2^*(\xi_\eta)|\geq \frac{1}{2}|E_2|$. We conclude the proof by sending $\eta\rightarrow1$.
\end{proof}
To prove our main result, we choose the sets $E_1$ and $E_2$ and the quantities $R_1$, $R_2$ and $S$ in such a way that they satisfy the assumptions of Property \ref{property 3}. Property \ref{property 3} helps us to shift the part of the integral on the left hand side of \eqref{eq main 1}, which is over $F$ to the integral over the complement of $F$. Property \ref{property 1} connects the integrals of both sides of \eqref{eq main 1} over the complement of $F$.\smallskip

The condition $|E|\geq Sx_1^{Q-1}|z|$ is crucial in Property \ref{property 2} and Property \ref{property 3}. To get this estimate we decompose the domain of integration as in the Dyda's  proof for the half space $\{x=(x_1,x_2,\ldots,x_n)\\ \in\RR^n:x_1>0\}$ in the Euclidean case. Dyda used dyadic decomposition for the base  and worked with the estimate $|x-y|\leq Rx_1$ for $x\in E_1$ and $y\in E_2$ and $|E_2|\geq Sx_1^n$ for some dyadic cubes $E_1$ and $E_2$, which is quite natural. In our case, if we were able to work with any decomposition for which $d({\xi}^{-1}\circ \xi')\leq Rx_1$ holds, then \eqref{eq main 1} would be true for any $\alpha\geq0$. Unfortunately, we don't have such a decomposition for $\HH^n$. For the dyadic type decomposition we use, we have instead $d({\xi}^{-1}\circ \xi')\leq Rx_1^{1/2}|z|^{1/2}$ for $\xi\in E_1$ and $\xi'\in E_2$ and $|E_2|\geq Sx_1^{Q}$, using these estimates we can prove Theorem \ref{thm 1.1} only for $\alpha\geq (Q+sp)/2$. We have slightly changed the decomposition in the $t$ direction and are able to prove Theorem \ref{thm 1.1} for $\alpha\geq (Q-2+sp)/2$. The decomposition is as follows.\smallskip

Let us define the set $\Omega_{x_1}=\{(x_2,\ldots,x_n,y,t):x_2,\ldots,x_n\in\RR,y\in\RR^n\ \mathrm{and}\ t\in\RR\}$. For a set $E\subset\Omega_{x_1}$ and for $\rho>0$ we define
\begin{equation}\label{set Q}
\mathcal{Q}(E,\rho)=\{\xi=(x,y,t)\in\HH^n_+:(x_2,\cdots,x_n,y,t)\in E\ \mathrm{and}\ \ 0<x_1<\rho\}.
\end{equation}
Let $K(\rho)=\{(x_2,\cdots,x_n,y,t):|x_i|,|y_j|<\rho/2,\ 2\leq i\leq n,\,1\leq j\leq n,\ \mathrm{and}\ |t|<\rho^2/2\}\subset\Omega_{x_1}$. We simply denote the set $\mathcal{Q}(K(\rho),\rho)$ as $\mathcal{Q}(\rho)$.\smallskip

Let $f\in C_c(\HH^n_+)$ and let $\rho$ be such that
$$\mathrm{supp}(f)\subset\{\xi\in\HH^n_+:x_1<\rho\ \mathrm{and}\ (x_2,\cdots,x_n,y,t)\in K(\rho)\}=\mathcal{Q}(\rho).$$
We write $K(\rho)=L(\rho)\times(-\rho^2/2,\rho^2/2)$, where $$L(\rho)=\{(x_2,\cdots,x_n,y_1,\cdots,y_n):|x_i|,|y_j|<\rho/2,\ 2\leq i\leq n,\,1\leq j\leq n\}.$$

\begin{center}
\begin{tikzpicture}
\draw[thick,->] (-4,0) -- (4.5,0) node[anchor=north west] {};
\draw[thick,->] (0,-4) -- (0,4.5) node[anchor=south east] {};
\draw[dashed] (-4,-3) -- (4,-3);
\draw[dashed] (-4,-2) -- (4,-2);
\draw[dashed] (-4,-1) -- (4,-1);
\draw[dashed] (-4,1) -- (4,1);
\draw[dashed] (-4,2) -- (4,2);
\draw[dashed] (-4,3) -- (4,3);
\draw[dashed] (-3,-4) -- (-3,4);
\draw[dashed] (-2,-4) -- (-2,4);
\draw[dashed] (-1,-4) -- (-1,4);
\draw[dashed] (1,-4) -- (1,4);
\draw[dashed] (2,-4) -- (2,4);
\draw[dashed] (3,-4) -- (3,4);
\draw[thick] (-1,-1) rectangle (1,1);
\draw[thick] (-2,-2) rectangle (2,2);
\draw[thick] (-3,-3) rectangle (3,3);
\draw[thick] (-4,-4) rectangle (4,4);
\draw (0.5,0.5) node{$L_{j,1}^{1}$};
\draw (0.5,-0.5) node{$L_{j,1}^{2}$};
\draw (-0.5,-0.5) node{$L_{j,1}^{3}$};
\draw (-0.5,0.5) node{$L_{j,1}^{4}$};
\draw (0.5,2.5) node{$L_{j,3}^{1}$};
\draw (1.5,2.5) node{$L_{j,3}^{2}$};
\draw (2.5,2.5) node{$\cdot$};
\draw (2.5,1.5) node{$\cdot$};
\draw (2.5,0.5) node{$\cdot$};
\draw (2.5,-0.5) node{$\cdot$};
\draw (2.5,-1.5) node{$\cdot$};
\draw (2.5,-2.5) node{$\cdot$};
\draw (1.5,-2.5) node{$\cdot$};
\draw (0.5,-2.5) node{$\cdot$};
\draw (-0.5,-2.5) node{$\cdot$};
\draw (-1.5,-2.5) node{$\cdot$};
\draw (-2.5,-2.5) node{$\cdot$};
\draw (-2.5,-1.5) node{$\cdot$};
\draw (-2.5,-0.5) node{$\cdot$};
\draw (-2.5,0.5) node{$\cdot$};
\draw (-2.5,1.5) node{$\cdot$};
\draw (-2.5,2.5) node{$\cdot$};
\draw (-1.5,2.5) node{$\cdot$};
\draw (-0.5,2.5) node{$\cdot$};
\end{tikzpicture}
\begin{tikzpicture}
\draw[thick,->] (-3,0) -- (3.5,0) node[anchor=north west] {};
\draw[thick,->] (0,-4) -- (0,4.5) node[anchor=south east] {t};
\draw[dashed] (-3,-4) -- (-3,4);
\draw[dashed] (3,-4) -- (3,4);
\draw[dashed] (-2,-4) -- (-2,4);
\draw[dashed] (-1,-4) -- (-1,4);
\draw[dashed] (1,-4) -- (1,4);
\draw[dashed] (2,-4) -- (2,4);

\draw[dashed] (-3,-3) -- (3,-3);
\draw[dashed] (-3,3) -- (3,3);
\draw[dashed] (-1,-2.5) -- (1,-2.5);
\draw[dashed] (-1,-2) -- (1,-2);
\draw[dashed] (-1,-1.5) -- (1,-1.5);
\draw[dashed] (-1,-1) -- (1,-1);
\draw[dashed] (-1,-0.5) -- (1,-0.5);
\draw[dashed] (-1,0.5) -- (1,0.5);
\draw[dashed] (-1,1) -- (1,1);
\draw[dashed] (-1,1.5) -- (1,1.5);
\draw[dashed] (-1,2) -- (1,2);
\draw[dashed] (-1,2.5) -- (1,2.5);

\draw[dashed] (-1,-3.5) -- (1,-3.5);
\draw[dashed] (-1,3.5) -- (1,3.5);

\draw[dashed] (-2,-2) -- (-1,-2);
\draw[dashed] (-2,-1) -- (-1,-1);
\draw[dashed] (-2,1) -- (-1,1);
\draw[dashed] (-2,2) -- (-1,2);
\draw[dashed] (1,-2) -- (2,-2);
\draw[dashed] (1,-1) -- (2,-1);
\draw[dashed] (1,1) -- (2,1);
\draw[dashed] (1,2) -- (2,2);

\draw[dashed] (-3,-1.5) -- (-2,-1.5);
\draw[dashed] (-3,1.5) -- (-2,1.5);
\draw[dashed] (2,-1.5) -- (3,-1.5);
\draw[dashed] (2,1.5) -- (3,1.5);

\draw (2.5,0.75) node{$E_{j,k}^{i,\ell}$};

\draw (2,0) -- (2.1,-0.22) -- (2.25,-0.22);
\draw (3,0) -- (2.9,-0.22) -- (2.7,-0.22);
\draw (2.5,-0.25) node{{\tiny $L_{j,k}^{i}$}};

\draw (3,1.5) -- (3.22,1.4) -- (3.22,1);
\draw (3,0) -- (3.22,0.1) -- (3.22,0.5);
\draw (3.25,0.75) node{{\tiny $I_k^\ell$}};
\draw[thick] (2,0) rectangle (3,1.5);
\end{tikzpicture}
\end{center}
$$\mathrm{Figure}\ 2\quad\qquad\qquad\qquad\qquad\qquad\qquad\qquad\qquad\mathrm{Figure}\ 3$$ \smallskip

\newpage For $j=0,1,2,\cdots$ we decompose $L(\rho)$ into $2^{j(2n-1)}$ sub-cubes with sides parallel to axes and each with side length $\rho/2^j$. We divide this sub-cubes into the classes $\{L_{j,1}^i\}$, $\{L_{j,2}^i\},\cdots,\{L_{j,2^{j-1}}^i\}$ in the following manner: a sub-cube $\tilde{L}$ belongs to the class $\{L_{j,k}^i\}$ if $\underset{\xi\in \tilde{L}}{\mathrm{sup}}\{|x_2|,\cdots,|x_n|,|y_1|,\cdots,|y_n|\}=\frac{k\rho}{2^j}$ (see Figure 2, where we have shown the classes for $k=1$ and $k=3$ for some fixed $j$).\smallskip

The non-isotropic guage requires dealing with  the $t-$variable differently, which we now do. Let $\ell\in\ZZ$. We consider the intervals $I_{j,k}^\ell=(\frac{\ell k\rho^2}{2^{2j}},\frac{(\ell+1)k\rho^2}{2^{2j}})$ for $k=1,2,\cdots,2^{j-1}$ along the $t$ direction and we define the sets $E_{j,k}^{i,\ell}=L_{j,k}^{i}\times I_{j,k}^\ell$ (see Figure 3). Note that, for each fixed $j=0,1,2, \ldots,$ $K(\rho)\subset\underset{k}{\cup}\big(\underset{i}{\cup}\big(\underset{\ell}{\cup}E_{j,k}^{i,\ell}\big)\big)$ upto a Lebesgue measure zero set.\smallskip

We denote by $\mathcal{Q}_{j,k}^{i,\ell}$, the domains $\mathcal{Q}(E_{j,k}^{i,\ell},\rho)$ in $\HH^n_+$ of same height $\rho$ in $x_1$ direction. Furthermore, for $j=0,1,\cdots$ we define the sets $\mathcal{A}_j=\mathcal{Q}(K(\rho),\rho/2^j)\setminus\mathcal{Q}(K(\rho),\rho/2^{j+1})$. Note that, $\mathcal{A}_j$'s are mutually disjoint and $\cup_{j=0}^\infty\mathcal{A}_j=\mathcal{Q}(\rho)$. Therefore the sets $\{\mathcal{A}_j\cap\mathcal{Q}_{j,k}^{i,\ell}\}$ are mutually disjoint upto a Lebesgue measure zero set, whose union over $\ell$, $i$, $k$ and $j$ respectively gives $\mathcal{Q}(\rho)$.\smallskip

The proof of Theorem \ref{thm 1.1} uses the following two lemmas. Lemma \ref{lemma 2.1} gives an important estimate for $d({\xi}^{-1}\circ \xi')$ when $\xi$ and $\xi'$ belong respectively to two sets at different heights in the $x_1$ direction but with the same base in the rest of the directions. The proof of Theorem \ref{thm 1.1} then follows from Lemma \ref{lemma 2.2}.\smallskip

\begin{lemma}\label{lemma 2.1}
Let $E_1=\mathcal{A}_j\cap\mathcal{Q}_{j,k}^{i,\ell}$ and $E_2=\mathcal{A}_{j+m}\cap\mathcal{Q}_{j,k}^{i,\ell}$ as defined above. Then $E_2\subset \Sigma_{R_1,R_2}(\xi)\cap\mathcal{Q}(\rho)$ and $|E_2|\geq Sx_1^{Q-1}|z|$ for all $\xi\in E_1$, where $R_1=2\sqrt{2}(n^2+n+2)^{\frac{1}{4}}+1$,  $R_2=2\sqrt{2n}+1$ and $S=2^{-m-1}/\sqrt{2n}$. In other words, $E_1$ and $E_2$ satisfy the assumptions of Property \ref{property 3}.
\end{lemma}
\begin{proof}
Clearly $E_1,E_2\subset\mathcal{Q}(\rho)$. Let $\xi\in E_1$. We have $\frac{\rho}{2^{j+1}}<x_1\leq\frac{\rho}{2^{j}}$, $\frac{(k-1)\rho}{2^{j+1}}<|z|\leq \sqrt{2n}\frac{k\rho}{2^{j}}$ if $k\geq2$ and $\frac{\rho}{2^{j+1}}<|z|\leq\sqrt{2n}\frac{\rho}{2^{j}}$ if $k=1$. Now, for any $\xi'\in E_2$ we have
\begin{multline*}
d({\xi}^{-1}\circ \xi')^4\leq|z'-z|^4+2(t'-t)^2+8(\langle y,x'\rangle-\langle x,y'\rangle)^2\\
=|z'-z|^4+2(t'-t)^2+8(\langle (y,-x),z'-z\rangle)^2\\
\leq|z'-z|^4+2(t'-t)^2+8|z|^2|z'-z|^2\qquad\quad\\
\leq4n^2\left(\frac{\rho}{2^j}\right)^4+2\left(\frac{k\rho^2}{2^{2j}}\right)^2+16n\left(\frac{\rho}{2^j}\right)^2|z|^2\\
\leq(4n^2+2k^2)\left(\frac{\rho}{2^j}\right)^4+16n\left(\frac{\rho}{2^j}\right)^2|z|^2.
\end{multline*}
For $k=1$,
$$d({\xi}^{-1}\circ \xi')^4\leq2^4(4n^2+2)\left(\frac{\rho}{2^{j+1}}\right)^4+2^4\cdot4n\left(\frac{\rho}{2^{j+1}}\right)^2|z|^2\leq 2^4(4n^2+4n+2)x_1^2|z|^2,$$
and for $k\geq2$, using the fact $k\leq2(k-1)$ we estimate
\begin{multline*}
d({\xi}^{-1}\circ \xi')^4\leq(4n^2+8(k-1)^2)\left(\frac{\rho}{2^j}\right)^4+16n\left(\frac{\rho}{2^j}\right)^2|z|^2\\
\leq2^4(4n^2+8)\left(\frac{\rho}{2^{j+1}}\right)^2\left(\frac{(k-1)\rho}{2^{j+1}}\right)^2+2^4\cdot4n\left(\frac{\rho}{2^{j+1}}\right)^2|z|^2\leq2^4(4n^2+4n+8)x_1^2|z|^2.
\end{multline*}
Therefore, we have, for any $\xi\in E_1$ and for any $\xi'\in E_2$,
$$d({\xi}^{-1}\circ \xi')\leq 2\sqrt{2}(n^2+n+2)^{\frac{1}{4}}x_1^{\frac{1}{2}}|z|^{\frac{1}{2}}<R_1x_1^{\frac{1}{2}}|z|^{\frac{1}{2}}.$$
Also, $|z'-z|\leq \sqrt{2n}\frac{\rho}{2^j}<R_2x_1$. These together imply $E_2\subset \Sigma_{R_1,R_2}(\xi)\cap\mathcal{Q}(\rho)$ for all $\xi\in E_1$.\smallskip

It is easy check that $|E_1|=2^m|E_2|$, and for $k\geq1$,
$$|E_2|=\frac{\rho}{2^{j+m+1}}\left(\frac{\rho}{2^j}\right)^{2n-1}\frac{k\rho^2}{2^{2j}}\geq\frac{2^{-m-1}}{\sqrt{2n}}x_1^{Q-1}|z|=Sx_1^{Q-1}|z|.$$
\end{proof}\smallskip

\begin{lemma}\label{lemma 2.2}
Let $p\geq1$ and $s\in(0,1)$ be such that $sp>1$ and $\alpha\geq (Q-2+sp)/2$. Then there exists a constant $C=C(n,p,s,\alpha)>0$ such that
\begin{equation}\label{eq 2.3}
\int_{\mathcal{Q}(\rho)}\frac{|f(\xi)|^p}{x_1^{sp}|z|^\alpha}d\xi\leq C\int_{\mathcal{Q}(\rho)}\int_{\mathcal{Q}(\rho)}\frac{|f(\xi)-f(\xi')|^p}{d({\xi}^{-1}\circ \xi')^{Q+sp}|z'-z|^\alpha}d\xi'd\xi\ \ \forall\,f\in C_c(\HH^n_+).
\end{equation}
\end{lemma}
\begin{proof}
Let $f\in C_c(\HH^n_+)$ and $\mathscr{D}=\mathcal{Q}(\rho)$. We consider the set $F=F(f,\mathscr{D},R_1,R_2,S)$ defined in \eqref{set F1} with $R_1=2\sqrt{2}(n^2+n+2)^{\frac{1}{4}}+1$,  $R_2=2\sqrt{2n}+1$ and $S=2^{-m-1}/\sqrt{2n}$. The number $m$ will be determined later.\smallskip

Fix $j=0,1,\cdots$ and let $E_1=\mathcal{A}_j\cap\mathcal{Q}_{j,k}^{i,\ell}$ and $E_2=\mathcal{A}_{j+m}\cap\mathcal{Q}_{j,k}^{i,\ell}$. 
Lemma \ref{lemma 2.1} says that $E_1$ and $E_2$ satisfy the assumptions of Property \ref{property 3}.
We have
$$\frac{\sup\{|z'|:\xi'\in E_2\}}{\inf\{|z|:\xi\in E_1\}}=\begin{cases}
\sqrt{2n}\frac{\rho}{2^j}/{\frac{\rho}{2^{j+1}}}=2\sqrt{2n} & \text{if $k=1$},\\
\\
\sqrt{2n}\frac{k\rho}{2^j}/{\frac{(k-1)\rho}{2^{j+1}}}\leq4\sqrt{2n} & \text{if $k\geq2$}.
\end{cases}$$
Using this in Property \ref{property 3} we get
\begin{multline}\label{eq 2.4}
\int_{(\mathcal{A}_j\cap\mathcal{Q}_{j,k}^{i,\ell})\cap F}\frac{|f(\xi)|^p}{x_1^{sp}|z|^\alpha}d\xi\\
\leq 2^{p+1}2^m\left(\frac{\frac{\rho}{2^{j+m}}}{\frac{\rho}{2^{j+1}}}\right)^{sp}(4\sqrt{2n})^\alpha\int_{\mathcal{A}_{j+m}\cap\mathcal{Q}_{j,k}^{i,\ell}}\frac{|f(\xi')|^p}{{x_1'}^{sp}|z'|^\alpha}d\xi'\\
=\gamma\int_{\mathcal{A}_{j+m}\cap\mathcal{Q}_{j,k}^{i,\ell}}\frac{|f(\xi')|^p}{{x_1'}^{sp}|z'|^\alpha}d\xi',
\end{multline}
where $\gamma=(\sqrt{2n})^\alpha2^{p+1+sp+2\alpha}2^{(1-sp)m}$. We choose $m$ to be large enough so that $\gamma<1$. Note that $\gamma$ is independent of j, this is crucial for iterating the argument. Summation over $\ell$, $i$ and $k$ respectively gives
\begin{equation}\label{eq 2.5}
\int_{\mathcal{A}_j\cap F}\frac{|f(\xi)|^p}{x_1^{sp}|z|^\alpha}d\xi\leq\gamma\int_{\mathcal{A}_{j+m}}\frac{|f(\xi')|^p}{{x_1'}^{sp}|z'|^\alpha}d\xi'.
\end{equation}
Next, we have
$$\int_{\mathcal{A}_{j+m}}\frac{|f(\xi')|^p}{{x_1'}^{sp}|z'|^\alpha}d\xi'=\left(\int_{\mathcal{A}_{j+m}\setminus F}+\int_{\mathcal{A}_{j+m}\cap F}\right)\frac{|f(\xi')|^p}{{x_1'}^{sp}|z'|^\alpha}d\xi',$$
and we iterate \eqref{eq 2.5} by replacing $j$ with $j+m,j+2m,\cdots,j+n_jm$ until the set $A_{j+n_jm}\cap F$ becomes empty. Therefore we obtain
\begin{multline*}
\int_{\mathscr{D}}\frac{|f(\xi)|^p}{{x_1}^{sp}|z|^\alpha}d\xi=\left(\int_{\mathscr{D}\setminus F}+\int_{\mathscr{D}\cap F}\right)\frac{|f(\xi)|^p}{{x_1}^{sp}|z|^\alpha}d\xi=\left(\int_{\mathscr{D}\setminus F}+\sum_{j=0}^\infty\int_{\mathcal{A}_j\cap F}\right)\frac{|f(\xi)|^p}{{x_1}^{sp}|z|^\alpha}d\xi\\
\leq \int_{\mathscr{D}\setminus F}\frac{|f(\xi)|^p}{{x_1}^{sp}|z|^\alpha}d\xi+\sum_{j=0}^\infty\left(\sum_{k=1}^\infty\gamma^k\int_{\mathcal{A}_{j+km}\setminus F}\frac{|f(\xi')|^p}{{x_1'}^{sp}|z'|^\alpha}d\xi'\right)\\
=\int_{\mathscr{D}\setminus F}\frac{|f(\xi)|^p}{{x_1}^{sp}|z|^\alpha}d\xi+\sum_{k=1}^\infty\gamma^k\cdot\sum_{j=0}^\infty\int_{\mathcal{A}_{j+km}\setminus F}\frac{|f(\xi')|^p}{{x_1'}^{sp}|z'|^\alpha}d\xi'\\
\leq \left(\sum_{k=0}^\infty\gamma^k\right)\int_{\mathscr{D}\setminus F}\frac{|f(\xi')|^p}{{x_1'}^{sp}|z'|^\alpha}d\xi'.
\end{multline*}
Applying Property \ref{property 1}, we get
$$\int_{\mathscr{D}}\frac{|f(\xi)|^p}{{x_1}^{sp}|z|^\alpha}d\xi\leq \frac{1}{1-\gamma}\frac{2^{p+1}R_1^{Q+sp}R_2^\alpha}{S}\int_{\mathscr{D}}\int_{\mathscr{D}}\frac{|f(\xi)-f(\xi')|^p}{d({\xi}^{-1}\circ \xi')^{Q+sp}|z'-z|\alpha}d\xi' d\xi.$$
This completes the proof.
\end{proof}

\begin{proof}[\textbf{Proof of Theorem \ref{thm 1.1}}]
Let $f\in C_c(\HH^n_+)$. We choose $\rho$ large enough such that $\mathrm{supp}(f)\subset\mathcal{Q}(\rho)$. Then \eqref{eq main 1} follows from \eqref{eq 2.3}. To be precise, inequality \eqref{eq 2.3} is stronger than \eqref{eq main 1}.
\end{proof}

\section{Alternate proof of the fractional Hardy's inequality}
Let $f\in C_c(\HH^n)$, $\mathscr{D}\subset \HH^n$ and $R,S>0$. Similar to \eqref{set F1}, we define the set
\begin{equation}\label{set F2}
G=G(f,\mathscr{D},R,S)=\left\{\xi\in\mathscr{D}:|f(\xi)|^p>\frac{2^{p+1}}{S|z|^Q}\int_{B(\xi,R|z|)\cap\mathscr{D}}|f(\xi)-f(\xi')|^pd\xi'\right\}.
\end{equation}
The set $G$ satisfies the following properties, which are similar to Property \ref{property 1}-\ref{property 3} of the set $F$ defined in \eqref{set F1}. The proofs are identical to those of Property \ref{property 1}-\ref{property 3}. We omit the proofs.
\begin{property}\label{property 4}
For $\alpha\geq0$ we have
$$\int_{\mathscr{D}\setminus G}\frac{|f(\xi)|^p}{|z|^{sp+\alpha}}d\xi\leq \frac{2^{p+1}R^{Q+sp+\alpha}}{S}\int_{\mathscr{D}\setminus G}\int_{\mathscr{D}}\frac{|f(\xi)-f(\xi')|^p}{d({\xi}^{-1}\circ \xi')^{Q+sp}|z'-z|^\alpha}d\xi' d\xi.$$
\end{property}

\begin{property}\label{property 5}
Let $\xi\in G$. For $E\subset B(\xi,R|z|)\cap\mathscr{D}$ we define the set
$$E^*(\xi):=\left\{\xi'\in E:\frac{1}{2}|f(\xi)|\leq |f(\xi')|\leq \frac{3}{2}|f(\xi)|\right\}.$$
If $|E|\geq S|z|^Q$, then $|E^*|\geq |E|-\frac{1}{2}S|z|^Q\geq \frac{1}{2}|E|$.
\end{property}

\begin{property}\label{property 6}
Let $E_1\subset\mathscr{D}$. If $E_2$ be a set such that $E_2\subset B(\xi,R|z|)\cap\mathscr{D}$ and $|E_2|\geq S|z|^Q$ for all $\xi\in E_1$, then
$$\int_{E_1\cap G}\frac{|f(\xi)|^p}{|z|^{sp+\alpha}}d\xi\leq 2^{p+1}\frac{|E_1|}{|E_2|}\left(\frac{\sup\{|z'|:\xi'\in E_2\}}{\inf\{|z|:\xi\in E_1\}}\right)^{sp+\alpha}\int_{E_2}\frac{|f(\xi')|^p}{|z'|^{sp+\alpha}}d\xi'.$$
\end{property}\smallskip

Let $r>0$ be fixed. For $j,k=0,1,2,\cdots$ and for $\ell\in\ZZ$ we define the sets
$$J_j^{k,\ell}=\{\xi\in\HH^n:2^jr<|z|<2^{j+1}r\ \mathrm{and}\ \ell\,2^{2k}r^2<t<(\ell+1)2^{2k}r^2\},$$
and $A_j=B_z(0,2^{j+1}r)\setminus B_z(0,2^{j}r)$. Note that, for any fixed $j$ and $k$, $\underset{\ell\in\ZZ}{\cup}J_j^{k,\ell}=A_j$ upto a Lebesgue measure zero set. The following two lemmas are the main ingredient to prove Theorem \ref{thm 1.2}.

\begin{lemma}\label{lemma 3.1}
Let $p\geq1$, $s\in(0,1)$ and $\alpha\geq0$ be such that $sp+\alpha<Q-2$ and let $r>0$. Then there exists a constant $C=C(n,p,s,\alpha)>0$ such that
\begin{equation}\label{eq 3.2}
\int_{B_z(0,r)^C}\frac{|f(\xi)|^p}{|z|^{sp+\alpha}}d\xi\leq C\int_{B_z(0,r)^C}\int_{B_z(0,r)^C}\frac{|f(\xi)-f(\xi')|^p}{d({\xi}^{-1}\circ \xi')^{Q+sp}|z'-z|^\alpha}d\xi' d\xi\ \ \forall\,f\in C_c(\HH^n).
\end{equation}
\end{lemma}
\begin{proof}
Let $\mathscr{D}=B_z(0,r)^C$. Let us consider the set $G=G(f,\mathscr{D},R,S)$ as \eqref{set F2} with $R=11^\frac{1}{4}2^{m+1}$ and $S=2^{Q(m-1)-2m}|B_{2n}(0,2)\setminus B_{2n}(0,1)|$, where $B_{2n}(0,r)$ is the ball in $\RR^{2n}$ centered at origin with radius $r$. The number $m$ will be determined later.\smallskip

Let $j\in\{0,1,2,\cdots\}$ be fixed and let $E_1=J_j^{j,\ell}$ and $E_2=J_{j+m}^{j,\ell}$. We first check that $E_1$ and $E_2$ satisfy the assumptions of Property \ref{property 6}.\\
Clearly $E_1,E_2\subset\mathscr{D}$. Let $\xi\in E_1$. We have $2^jr<|z|<2^{j+1}r$, and for any $\xi'\in E_2$,
\begin{multline}\label{eq 3.3}
d({\xi}^{-1}\circ \xi')\leq(|z'-z|^4+2(t'-t)^2+8|z|^2|z'-z|^2)^\frac{1}{4}\\
\leq ((2^{j+m+1}r)^4+2(2^{2j}r^2)^2+8(2^{j+m+1}r)^2(2^{j+m+1}r)^2)^\frac{1}{4}<11^\frac{1}{4}2^{j+m+1}r<R|z|,
\end{multline}
which implies $E_2\subset B(\xi,R|z|)\cap\mathscr{D}$ for all $\xi\in E_1$.\\
Also, one can easily check that $|E_1|=2^{-(Q-2)m}|E_2|$ and
$$|E_2|=(2^{j+m}r)^{Q-2}|B_{2n}(0,2)\setminus B_{2n}(0,1)|\,2^{2j}r^2=S(2^{j+1}r)^Q>S|z|^Q.$$
Hence from property \ref{property 6} we get,
$$\int_{J_j^{j,\ell}\cap G}\frac{|f(\xi)|^p}{|z|^{sp+\alpha}}d\xi\\
\leq 2^{p+1} 2^{-(Q-2)m}\left(\frac{2^{j+m+1}}{2^j}\right)^{sp+\alpha}\int_{J_{j+m}^{j,\ell}}\frac{|f(\xi')|^p}{|z'|^{sp+\alpha}}d\xi'=\gamma\int_{J_{j+m}^{j,\ell}}\frac{|f(\xi')|^p}{|z'|^{sp+\alpha}}d\xi',$$
where $\gamma=2^{(sp+\alpha-Q+2)m}\,2^{p+1+sp+\alpha}$. We choose $m$ to be large enough so that $\gamma<1$. Summation over $\ell$ gives
\begin{equation}
\int_{A_j\cap G}\frac{|f(\xi)|^p}{|z|^{sp+\alpha}}d\xi\leq\gamma\int_{A_{j+m}}\frac{|f(\xi')|^p}{|z'|^{sp+\alpha}}d\xi'.
\end{equation}
The rest of the proof proceeds analogously to the argument concluding Lemma \ref{lemma 2.2}.
\end{proof}

\begin{lemma}\label{lemma 3.2}
Let $p\geq1$, $s\in(0,1)$ and $\alpha\geq0$ be such that $sp+\alpha>Q-2$ and let $r>0$. Then there exists a constant $C=C(n,p,s,\alpha)>0$ and a natural number $m=m(n,p,s)$ such that
\begin{equation}\label{eq 3.5}
\int_{B^C}\frac{|f(\xi)|^p}{|z|^{sp+\alpha}}d\xi\leq C\left(\int_{B^C}\int_{B^C}\frac{|f(\xi)-f(\xi')|^p}{d({\xi}^{-1}\circ \xi')^{Q+sp}|z'-z|^\alpha}d\xi' d\xi+\int_{A}\frac{|f(\xi)|^p}{|z|^{sp+\alpha}}d\xi\right)\ \ \forall\,f\in C_c(\HH^n),
\end{equation}
where $B=B_z(0,2^mr)$ and $A=B_z(0,2^mr)\setminus B_z(0,r)$.
\end{lemma}
\begin{proof}
As in the proof of previous lemma, we assume $f\in C_c(\HH^n)$ and $\mathscr{D}=B_z(0,r)^C$. We consider the set $G=G(f,\mathscr{D},R,S)$ with $R=2\cdot11^{\frac{1}{4}}$ and $S=2^{2m-Q(m+1)}|B_{2n}(0,2)\setminus B_{2n}(0,1)|$. The number $m$ will be determined later. For $j,k=0,1,2,\cdots$ and for $\ell\in\ZZ$, consider the sets $J_j^{k,\ell}$ and $A_j$ as in Lemma \ref{lemma 3.1}. \smallskip

Let $j\geq m$ be fixed and let $E_1=J_j^{j,\ell}$ and $E_2=J_{j-m}^{j,\ell}$. Then  $E_1,E_2\subset\mathscr{D}$. For any $\xi\in E_1$, we have $2^jr<|z|<2^{j+1}r$, and for any $\xi'\in E_2$,
\begin{multline}
d({\xi}^{-1}\circ \xi')\leq(|z'-z|^4+2(t'-t)^2+8|z|^2|z'-z|^2)^\frac{1}{4}\\
\leq ((2^{j+1}r)^4+2(2^{2j}r^2)^2+8(2^{j+1}r)^2(2^{j+1}r)^2)^\frac{1}{4}<11^\frac{1}{4}2^{j+1}r<R|z|,
\end{multline}
which implies $E_2\subset B(\xi,R|z|)\cap\mathscr{D}$ for all $\xi\in E_1$.\\
Also, one can easily check that $|E_1|=2^{(Q-2)m}|E_2|$ and
$$|E_2|=(2^{j-m}r)^{Q-2}|B_{2n}(0,2)\setminus B_{2n}(0,1)|\,2^{2j}r^2=S(2^{j+1}r)^Q>S|z|^Q.$$
Therefore, by Property \ref{property 6} we get
$$\int_{J_j^{j,\ell}\cap G}\frac{|f(\xi)|^p}{|z|^{sp+\alpha}}d\xi\leq 2^{p+1}2^{(Q-2)m}\left(\frac{2^{j-m+1}}{2^j}\right)^{sp+\alpha}\int_{J_{j-m}^{j,\ell}}\frac{|f(\xi')|^p}{|z'|^{sp+\alpha}}d\xi'=\gamma\int_{J_{j-m}^{j,\ell}}\frac{|f(\xi')|^p}{|z'|^{sp+\alpha}}d\xi',$$
where $\gamma=2^{(Q-2-sp-\alpha)m}2^{p+1+sp+\alpha}$. We choose $m$ to be large enough so that $\gamma<1$. Summation over $\ell$ gives
\begin{equation}\label{eqn 4.7}
\int_{A_j\cap G}\frac{|f(\xi)|^p}{|z|^{sp+\alpha}}d\xi\leq\gamma\int_{A_{j-m}}\frac{|f(\xi')|^p}{|z'|^{sp+\alpha}}d\xi'.
\end{equation}
Next, we have
$$\int_{A_{j-m}}\frac{|f(\xi')|^p}{|z'|^{sp+\alpha}}d\xi'=\left(\int_{A_{j-m}\setminus G}+\int_{A_{j-m}\cap G}\right)\frac{|f(\xi')|^p}{|z'|^{sp+\alpha}}d\xi',$$
and we iterate \eqref{eqn 4.7} by replacing $j$ with $j-m,j-2m,\cdots,j-k_jm$ until $0\leq j-k_jm<m$. Note that, $k_j=k$ for $j\in\{km,km+1,\ldots,km+m-1\}$ and $\cup_{j=km}^{km+m-1}A_{j-k_jm}\subset A$. Therefore we obtain
\begin{multline*}
\int_{B^C}\frac{|f(\xi)|^p}{|z|^{sp+\alpha}}d\xi=\left(\int_{B^C\setminus G}+\sum_{j=m}^\infty\int_{A_j\cap G}\right)\frac{|f(\xi)|^p}{|z|^{sp+\alpha}}d\xi\\
\leq \int_{B^C\setminus G}\frac{|f(\xi)|^p}{|z|^{sp+\alpha}}d\xi+\sum_{j=m}^\infty\left(\sum_{k=1}^{k_j-1}\gamma^k\int_{A_{j-km}\setminus G}\frac{|f(\xi')|^p}{|z'|^{sp+\alpha}}d\xi'+\gamma^{k_j}\int_{A_{j-k_jm}\setminus G}\frac{|f(\xi')|^p}{|z'|^{sp+\alpha}}d\xi'\right)\\
\leq \int_{B^C\setminus G}\frac{|f(\xi)|^p}{|z|^{sp+\alpha}}d\xi+\left(\sum_{k=1}^\infty\gamma^k\right)\sum_{j=m}^\infty\int_{A_{j-km}\setminus G}\frac{|f(\xi')|^p}{|z'|^{sp+\alpha}}d\xi'+\left(\sum_{k=1}^\infty\gamma^k\right)\int_{A\setminus G}\frac{|f(\xi')|^p}{|z'|^{sp+\alpha}}d\xi'\\
=\left(\sum_{k=0}^\infty\gamma^k\right)\int_{B^C\setminus G}\frac{|f(\xi')|^p}{|z'|^{sp+\alpha}}d\xi'+\left(\sum_{k=1}^\infty\gamma^k\right)\int_{A\setminus G}\frac{|f(\xi')|^p}{|z'|^{sp+\alpha}}d\xi'.
\end{multline*}
Applying Property \ref{property 4}, we get
$$\int_{B^C}\frac{|f(\xi)|^p}{|z|^{sp+\alpha}}d\xi\leq \frac{1}{1-\gamma}\frac{2^{p+1}R^{Q+sp+\alpha}}{S}\int_{B^C}\int_{B^C}\frac{|f(\xi)-f(\xi')|^p}{d({\xi}^{-1}\circ \xi')^{Q+sp}|z'-z|^\alpha}d\xi' d\xi+\frac{\gamma}{1-\gamma}\int_{A}\frac{|f(\xi')|^p}{|z'|^{sp+\alpha}}d\xi'.$$
This completes the proof.
\end{proof}

\begin{proof}[\textbf{Proof of Theorem \ref{thm 1.2}}]
We first consider the case $sp+\alpha<Q-2$ and assume $f\in C_c(\HH^n)$. By Lemma \ref{lemma 3.1}, we get for any $r>0$,
\begin{multline}\label{eqn 4.8}
\int_{B_z(0,r)^C}\frac{|f(\xi)|^p}{d(\xi)^{sp}|z|^{\alpha}}d\xi\leq \int_{B_z(0,r)^C}\frac{|f(\xi)|^p}{|z|^{sp+\alpha}}d\xi\\
\leq C\int_{B_z(0,r)^C}\int_{B_z(0,r)^C}\frac{|f(\xi)-f(\xi')|^p}{d({\xi}^{-1}\circ \xi')^{Q+sp}|z'-z|^\alpha}d\xi' d\xi\\
\leq C\int_{\HH^n}\int_{\HH^n}\frac{|f(\xi)-f(\xi')|^p}{d({\xi}^{-1}\circ \xi')^{Q+sp}|z'-z|^{\alpha}}d\xi'd\xi.
\end{multline}
Inequality \eqref{eq adi} follows by passing to the limit $r\rightarrow0$ in \eqref{eqn 4.8}.\smallskip

Next, we consider the case $sp+\alpha>Q-2$. Assume that $f\in C_c(\HH^n\setminus\{\xi\in\HH^n:z=0\})$. Consider the constant $m$ that appears in Lemma \ref{lemma 3.2}, which is independent of $r$. We choose $r$ small enough such that $\mathrm{supp}(f)\cap B_z(0,2^mr)=\phi$. Then \eqref{eq adi} follows from \eqref{eq 3.5}.
\end{proof}

\section{Appendix}
\subsection{An estimate of the distance function}
Let $\xi=(x,y,t)\in\HH^n_{+,t}$ and consider the point $\xi'=(x,y,0)\in\partial\HH^n_{+,t}$. Clearly, $d({\xi}^{-1}\circ \xi')=\sqrt{t}$, which says that $\delta_{\HH^n_{+,t}}(\xi)\leq\sqrt{t}$.\\
Next, for any point $\xi=(x,y,t)\in\HH^n_{+,t}$ with $x_i\neq0$ and $y_i\neq0$, $i=1,2,\cdots,n$, we consider the point $\xi_0=\left(x_1-t/{4ny_1},\ldots,x_n-t/{4ny_n},y_1+t/{4nx_1},\ldots,y_n+t/{4nx_n},0\right)\in\partial\HH^n_{+,t}.$
We have
\begin{multline*}
d({\xi}^{-1}\circ \xi_0)^4=\left(\sum_{i=1}^n\left(\frac{t}{4ny_i}\right)^2+\sum_{i=1}^n\left(\frac{t}{4nx_i}\right)^2\right)^2\\
+\left(-t-2\left\langle y,\left(x_1-\frac{t}{4ny_1},\cdots,x_n-\frac{t}{4ny_n}\right)\right\rangle+2\left\langle x,\left(y_1+\frac{t}{4nx_1},\cdots,y_n+\frac{t}{4nx_n}\right)\right\rangle\right)^2\\
=\left(\frac{t}{4n}\right)^4\left(\sum_{i=1}^n\left(\frac{1}{x_i^2}+\frac{1}{y_i^2}\right)\right)^2.
\end{multline*}
Therefore,
\begin{equation}
\delta_{\HH^n_{+,t}}(\xi)\leq \frac{t}{4n}\left(\sum_{i=1}^n\left(\frac{1}{x_i^2}+\frac{1}{y_i^2}\right)\right)^{\frac{1}{2}}.
\end{equation}
In particular, $\delta_{\HH^n_{+,t}}(\xi)<\sqrt{t}$, if $\xi$ belongs to the set
\begin{equation}\label{set Omega}
\Omega=\{\xi=(x,y,t)\in\HH^n:0<t<1,\ \ |x_i|,|y_i|>1/2\sqrt{2n},\ \ i=1,\cdots,n\}.
\end{equation}\smallskip

\textbf{Acknowledgements:} We acknowledge the Department of Mathematics and Statistics at IIT Kanpur for their invaluable support in providing a research-friendly environment. This work is part of doctoral thesis of the second author. He is grateful for the support provided by IIT Kanpur, India and MHRD, Government of India (GATE fellowship).

\end{document}